\documentclass[11pt,letterpaper,reqno]{amsart}

\pdfoutput=1

\usepackage{fullpage}
\usepackage[foot]{amsaddr}
\usepackage{amsmath}
\usepackage{amsthm}
\usepackage{amssymb}
\usepackage{mathrsfs}
\usepackage{url}

\usepackage{todonotes}
\usepackage{mathtools}
\usepackage{textpos}
\usepackage[T1]{fontenc}

\usepackage[breaklinks,bookmarks=true,hypertexnames=false,pagebackref]{hyperref}
\hypersetup{colorlinks=true, citecolor=blue, linkcolor=red, urlcolor=blue}
\usepackage{cleveref}
\usepackage{bm}

\usepackage{thm-restate}

\makeatletter
\def\paragraph{\@startsection{paragraph}{4}%
  \z@\z@{-\fontdimen2\font}%
  {\normalfont\bfseries}}
\makeatother

\newtheorem{theorem}{Theorem}

\newtheorem{lemma}[theorem]{Lemma}

\newtheorem{proposition}[theorem]{Proposition}
\newtheorem{corollary}[theorem]{Corollary}
\theoremstyle{definition}

\theoremstyle{remark}

\AddToHook{env/lemma/begin}{\crefalias{theorem}{lemma}}
\AddToHook{env/claim/begin}{\crefalias{theorem}{claim}}
\AddToHook{env/observation/begin}{\crefalias{theorem}{observation}}
\AddToHook{env/proposition/begin}{\crefalias{theorem}{proposition}}
\AddToHook{env/corollary/begin}{\crefalias{theorem}{corollary}}
\AddToHook{env/condition/begin}{\crefalias{theorem}{condition}}
\AddToHook{env/definition/begin}{\crefalias{theorem}{definition}}
\AddToHook{env/question/begin}{\crefalias{theorem}{question}}
\AddToHook{env/example/begin}{\crefalias{theorem}{example}}
\AddToHook{env/conjecture/begin}{\crefalias{theorem}{conjecture}}
\AddToHook{env/remark/begin}{\crefalias{theorem}{remark}}

\crefname{theorem}{Theorem}{Theorems}
\crefname{observation}{Observation}{Observations}
\crefname{claim}{Claim}{Claims}
\crefname{condition}{Condition}{Conditions}
\crefname{algorithm}{Algorithm}{Algorithms}
\crefname{property}{Property}{Properties}
\crefname{example}{Example}{Examples}
\crefname{fact}{Fact}{Facts}
\crefname{lemma}{Lemma}{Lemmas}
\crefname{corollary}{Corollary}{Corollaries}
\crefname{definition}{Definition}{Definitions}
\crefname{remark}{Remark}{Remarks}
\crefname{proposition}{Proposition}{Propositions}
\crefname{section}{Section}{Sections}
\crefname{equation}{equation}{equations}

\newcommand{\numP}{\#{\textnormal{\textbf{P}}}}
\newcommand{\NP}{\textnormal{\textbf{NP}}}

\newcommand{\modPp}{\textnormal{\textbf{modP}}_p}

\newcommand{\defeq}{\coloneqq}

\DeclareMathOperator{\tw}{tw}

\title{homomorphic-core phase transition threshold in Erd\H{o}s--R\'{e}nyi random graphs}

\author{Jiaheng Wang}
\address[Jiaheng Wang]{University of Helsinki, Finland. \textnormal{Financially supported by} Helsinki Institute of Information Technology (HIIT).}

\begin{document}

\begin{abstract}
It is shown in this manuscript that a random graph $G$ drawn from the Erd\H{o}s--R\'{e}nyi model $\mathcal{G}(n,p)$ with
\[
p=p(n)\leq 1/2, \qquad \lim_{n\to+\infty}(np-\log n-\log\log n)=+\infty,
\]
is a homomorphic core, i.e., every homomorphism from $G$ to itself is an automorphism. 
This implies tight ETH-based lower bounds of the subgraph isomorphism problem for almost all $k$-vertex patterns with polynomial average degree. 
\end{abstract}

\maketitle

{
\small \noindent \textsc{Disclosure of AI tool usage.} The core\footnote{This is a pun. The author is aware that explaning a joke is the best way to ruin the joke, however.} idea of the proof was found by ChatGPT 5.6 Sol, 
while the manuscript is wrtten by the \emph{homo sapiens} author in a traditional manner. 
The author shall \emph{forgo} claiming intellectual property of the key technical contribution in this manuscript. 
The author, however, takes full responsibility of endorsing the correctness and proper citations. 

The total number of hours that the model worked on this task was approximately 7 hours, summing up the cost all interactions. 
The model ran at xHigh and Max levels. 
This information is disclosed on the basis of potential concern with fair and eco-friendly usage of AI tools. 

The disclosure is intentionally put before the main text of this manuscript for visibility. 
}

\medskip

\section{Introduction}

\subsection{Motivation: the complexity of homomorphisms}

Graph homomorphisms \cite{MR3012035,MR2089014} play a fundamental role in theoretical computer science because they provide a uniform language for expressing structure-preserving assignments. 
Many seemingly unrelated problems, including graph colouring, database-query evaluation, and constraint satisfaction problems, can be formulated as asking whether a graph homomorphism exists and how many such homomorphisms there are \cite{DBLP:conf/stoc/ChandraM77,DBLP:conf/stoc/Schaefer78,DBLP:journals/siamcomp/FederV98,DBLP:journals/jcss/KolaitisV00,DBLP:journals/talg/FockeGRZ25}. 
This motivates the study of the complexity of deciding whether graph homomorphisms exist and of counting them. 
Recall that a graph homomorphism from $G$ to $H$ is a map from $V(G)$ to $V(H)$ that preserves adjacency. We denote by $\#\hom(G\to H)$ the number of homomorphisms from $G$ to $H$.

We focus on homomorphism problems \emph{from the left}, in which the host graph $H$ is the input.
The task $\textsc{Hom}(G)$ asks whether such a homomorphism exists, and $\textsc{\#Hom}(G)$ asks how many exist. 
These problems are naturally \emph{parameterised} by $G$. 
By contrast, when $H$ is fixed and $G$ is the input, complexity dichotomies are typically achieveable: 
depending on whether the host graph $H$ satisfies certain criteria, the problems are either solvable in polynomial time or hard with respect to their complexity classes ($\NP$-hard, $\numP$-hard, $\modPp$-hard); see for example \cite{MR1047555,DBLP:journals/rsa/DyerG00,DBLP:journals/jacm/CaiC17,DBLP:journals/toct/GalanisGJ17,DBLP:journals/toct/GobelLS21,DBLP:conf/stoc/BulatovK22,DBLP:conf/focs/Bulatov17,DBLP:journals/jacm/Zhuk20}.

On $n$-vertex input graphs, $\textsc{\#Hom}(G)$ can always be solved in time $n^{\Theta(|V(G)|)}$ by brute-force enumeration. 
For certain graphs, this running time is optimal. 
If $G$ is a $k$-clique, then $\textsc{Hom}(G)$ (respectively, its counting variant) is equivalent to the \emph{subgraph isomorphism problem} $\textsc{Sub}(G)$, which asks whether the input graph contains a subgraph isomorphic to $G$ (not necessarily an induced one). 
The Exponential-Time Hypothesis (ETH) \cite{DBLP:journals/jcss/ImpagliazzoP01,DBLP:journals/jcss/ImpagliazzoPZ01} implies an $n^{\Omega(k)}$-time lower bound for $\textsc{Sub}(G)$ when $G$ is a $k$-clique \cite{DBLP:journals/iandc/ChenCFHJKX05,DBLP:journals/jcss/ChenHKX06}.

Now consider the case in which $G$ is the $k$-by-$k$ biclique $K_{k,k}$ (the complete bipartite graph). 
Although the counting problem $\textsc{\#Hom}(K_{k,k})$ is intractable under standard complexity assumptions \cite{DBLP:journals/siamcomp/FlumG04,DBLP:journals/tcs/DalmauJ04}, its decision counterpart behaves very differently. 
In fact, the decision problem $\textsc{Hom}(K_{k,k})$ is trivial. 
Indeed, if $H$ contains at least one edge $uv$, then there is a homomorphism from $K_{k,k}$ to $H$: map every vertex in its left part to $u$ and every vertex in its right part to $v$. Otherwise, $H$ has no edges, and no homomorphism from $K_{k,k}$ to $H$ exists.

This argument extends to broader graph families. 
Suppose that a pattern graph $G$ has a \emph{proper} subgraph $G'$ and that there exists a homomorphism $f$ from $G$ to $G'$. 
If the host graph $H$ contains a homomorphic copy of $G$, then restricting this copy to $G'$ yields a homomorphic copy of $G'$. 
Conversely, if $H$ contains a homomorphic copy of $G'$ through a mapping $g$, then the composition $g\circ f$ is a homomorphism from $G$ to $H$, giving a homomorphic copy of $G$. 
Therefore, $\textsc{Hom}(G)$ and $\textsc{Hom}(G')$ are essentially equivalent. 
In other words, whereas the complexity of $\textsc{\#Hom}(G)$ is governed by the \emph{treewidth} of $G$ \cite{DBLP:journals/siamcomp/FlumG04,DBLP:journals/tcs/DalmauJ04}, 
the complexity of $\textsc{Hom}(G)$ depends on the treewidth of the \emph{homomorphic core} of $G$ rather than on that of $G$ itself (see the references \cite{DBLP:journals/tcs/ChekuriR00,DBLP:conf/aaai/Freuder90,DBLP:conf/cp/DalmauKV02} for algorithms and the work by Grohe \cite{DBLP:journals/jacm/Grohe07} for the lower bound).

 Recall that a homomorphic core\footnote{We refrain from abbreviating this term as merely \emph{core}, because a core may also refer to a $k$-core in the sense of graph degeneracy. Graph degeneracy will be used later in the proof.} is a graph that admits no homomorphism to any proper subgraph. 
 The homomorphic core of a graph $G$ is the smallest subgraph of $G$ that is itself a core and to which $G$ admits a homomorphism. 
 Every graph has a unique homomorphic core up to isomorphism \cite{MR1192374}, providing an alternative characterisation of the useful, intensively studied notion of \emph{homomorphic equivalence}.

\subsection{Motivation: the complexity of subgraph isomorphism}

Thus far, we have seen how homomorphic cores help us understand the complexity of homomorphism problems. 
They also turn out to be useful in the context of subgraph isomorphism and counting. 
See \cite{DBLP:conf/stoc/CurticapeanDM17} for a framework that connects subgraph counts and homomorphism counts.  

The computational task $\textsc{Sub}(G)$ asks whether an input graph contains $G$ as a subgraph.
A well-studied variant is the \emph{colourful} problem $\textsc{ColSub}(G)$, defined as follows. 
Let $G$ have $k$ vertices labelled by $1,\dots,k$. 
The input consists of a graph $H$ together with a partition of its vertex set, $V(H)=\uplus_{i\in[k]}V_i$. 
The problem asks whether there exist vertices $v_1,\dots,v_k$ with $v_i\in V_i$ such that $v_iv_j\in E(H)$ for every $ij\in E(G)$. 
Intuitively, in the colourful setting, the vertices of the pattern graph $G$ receive distinct colours, the input graph $H$ is coloured using the same colours, and the embedding must preserve these colours. 
Equivalently, $\textsc{ColSub}(G)$ can be viewed as a $2$-CSP (a constraint satisfaction problem of arity $2$), with variables $x_i$ for $i\in[k]$, relations $R_e$ for $e\in E(G)$, and domains $\mathcal{D}(x_i)=V_i$.

The complexity of $\textsc{ColSub}(G)$ is much better understood than that of its uncoloured counterpart $\textsc{Sub}(G)$. 
For $\textsc{ColSub}(G)$, the treewidth of $G$ fully governs the complexity in a nearly fine-grained sense: if $\tw(G)=t$, dynamic programming over a tree decomposition solves $\textsc{ColSub}(G)$ in time $O(n^{\Theta(t)})$; conversely, assuming ETH, there exists a universal constant $C>0$ such that, for every fixed graph $G$ of treewidth $t$, $\textsc{ColSub}(G)$ admits no $O(n^{C\cdot(t/\log t)})$-time algorithm \cite{DBLP:journals/toc/Marx10,DBLP:conf/sosa/SMPS24,DBLP:journals/jcss/CurticapeanDNW26}. 

The above lower bound for $\textsc{ColSub}(G)$ can be transferred to the uncoloured problem $\textsc{Sub}(G)$ whenever $G$ is a homomorphic core.
When $G$ is not a homomorphic core, however, the picture becomes far less clear \cite{DBLP:conf/stacs/MarxP14}.  
This difficulty arises, in particular, for non-trivial bipartite graphs. 
Even for highly specific pattern families, determining the complexity has been a major challenge, let alone obtaining a complete complexity dichotomy for $\textsc{Sub}(G)$. 
For instance, the parameterised complexity of the $k$-biclique problem was resolved only relatively recently \cite{DBLP:journals/jacm/Lin18}, even though, as a standard textbook on parameterised complexity puts it, ``almost everyone considers that this problem should \emph{obviously} be $\mathbf{W}[1]$-hard'' \cite{MR3154461}. 
The same was true for grid graphs \cite{DBLP:conf/wg/ChenGL17}. 
It is worth noting that both papers received best-paper awards at their respective conferences.

Thus, $\textsc{Sub}(G)$ is a challenging problem to study for \emph{every} $G$. Nevertheless, we may study its complexity for a \emph{typical} pattern drawn from the Erd\H{o}s--R\'{e}nyi random graph model $\mathcal{G}(k,p)$.

For $\textsc{ColSub}(G)$, this setting is well understood when the edge probability lies slightly above the phase-transition threshold for connectivity; 
that is, when $p\geq(1+\varepsilon)\log k/k$ for any fixed $\varepsilon>0$. 
In this regime, let $\rho=\log(k)/\log(kp)$ denote the \emph{typical diameter}. 
Then, with high probability over a random graph $G\sim\mathcal{G}(k,p)$, one obtains an $\Omega(n^{C\cdot k/\rho})$ lower bound for $\textsc{ColSub}(G)$ \cite{DBLP:journals/jcss/CurticapeanDNW26}. 
To transfer this lower bound to the uncoloured setting, it would therefore suffice to show that a random pattern is a homomorphic core with high probability in the desired regime.

\subsection{Main results}

The first result in this direction was established by Hell and Ne{\v s}et{\v r}il \cite{MR2089014}, who showed that a random graph drawn from $\mathcal{G}(n,1/2)$ is a homomorphic core with high probability. 
For simplicity, we assume $p\leq 1/2$ throughout the remainder of the discussion. 
Bonato and Pra{\l}at \cite{MR2567955} extended this result to the regime $p>n^{-1/3}\log^2 n$. 
(See also the work of Parsa and Kayll \cite{MR4224391} on the directed version.)

On the other hand, observe that any $n$-vertex graph ($n\geq 2$) with at least one isolated vertex is not a homomorphic core: the isolated vertex can be mapped to any other vertex of the graph. 
Furthermore, any $n$-vertex graph ($n\geq 3$) with a degree-$1$ vertex $v$ is not a homomorphic core. 
Indeed, if the vertex adjacent to $v$ has degree at least $2$, then $v$ can be mapped to another neighbour of that vertex. 
Consequently, in any regime in which the typical \emph{minimum degree} is $0$ or $1$, a random graph is not a core with high probability.

The phase-transition threshold for minimum degree at least $2$ is $p=(\log n+\log\log n\pm\omega(1))/n$ \cite{MR130187}, so the phase-transition threshold for being a homomorphic core cannot be lower. 
It was conjectured that the phase-transition threshold for being a homomorphic core coincides with that for minimum degree at least $2$ \cite{MR2567955}. Our main result confirms this conjecture.

\begin{theorem}\label{thm:main}
Let $p=p(n)\leq 1/2$. 
If
\[
\lim_{n\to+\infty}(np-\log n-\log\log n)=+\infty, 
\]
then with high probability $G\sim\mathcal{G}(n,p)$ is a homomorphic core. 
Conversely, if
\[
\lim_{n\to+\infty}(np-\log n-\log\log n)=-\infty, 
\]
then with high probability $G\sim\mathcal{G}(n,p)$ is not a homomorphic core. 
\end{theorem}

By the discussion in the previous subsection, our result yields a lower bound for the subgraph isomorphism problem $\textsc{Sub}(G)$ when the pattern $G\sim\mathcal{G}(k,p)$ is random and $p\geq(1+\varepsilon)\log k/k$.

\begin{corollary}
Assuming ETH, there exists an absolute constant $c>0$ such that, for every constant $\varepsilon>0$ and every $p\geq (1+\varepsilon)\log k/k$, the following holds: 
with high probability over an Erd\H{o}s--R\'{e}nyi random graph $H\sim\mathcal{G}(k,p)$, 
the problem $\textnormal{\textsc{Sub}}(H)$ admits no $O(n^{c\cdot k/\rho})$-time algorithm on $n$-vertex input graphs. 
Here, $\rho=\log(k)/\log(kp)$ denotes the typical distance in $\mathcal{G}(k,p)$ \cite{MR1826308}.
\end{corollary}

In particular, if $p=k^{-\alpha}$ is an inverse polynomial in $k$ with a fixed $\alpha>0$ and lies within the regime of the corollary, then the typical distance $\rho$ is bounded by a constant. 
In this case, the corollary implies that, under ETH and with high probability, no $O(n^{c'\cdot k})$-time algorithm exists, for some constant $c'>0$ depending only on $\alpha$ and $c$. 
In other words, almost every pattern with polynomial average degree is \emph{maximally} hard: unless ETH fails, no algorithm can fundamentally beat the $n^{O(k)}$-time brute-force algorithm.

\subsection{Proof overview}

An equivalent definition of homomorphic core is through the \emph{retraction}. 
A homomorphism $r$ from a graph $G$ to its induced subgraph $G[S_r]$ with image $S_r$ is a retraction if every element in $S_r$ is a fixed point, i.e., $r(V(G))=S_r$ and $r(s)=s$ for every $s\in S_r$. 
It is \emph{proper} if $S_r\neq V(G)$. 
It can be shown that a graph is a homomorphic core if and only if it has no proper retraction. 
We therefore prove the upper-threshold direction by showing that, with high probability, no map onto a proper induced subgraph can be a retraction. 

\medskip

\paragraph{The above-critical regime.} (\Cref{sec:above-critical}) 
We first consider the simpler regime in which the expected degree $np$ is much larger than $\log n$. 
For a fixed candidate retraction $r$, one can view the drawing procedure of a randomly graph $G$ in two stages: we reveal the graph induced by its image $S_r$ before revealing the remaining edges. 
Conditional on the revealed subgraph, all other edges remain independent. 
Any pair that the candidate maps to a non-edge must itself be absent, so the probability that the candidate is a retraction is controlled by the number of edges it forces to be absent.

We condition on typical bounds for the \emph{maximum degree} and \emph{clique number} of the random graph. 
If the image of the candidate is large, the maximum-degree bound implies that every vertex outside the image forces many edges across the boundary to be absent. 
This probability penalty dominates the number of possible maps. 
If the image is small, the clique-number bound, together with the \emph{Motzkin--Straus inequality}, shows that many pairs outside the image must be absent. 
The resulting penalty is strong enough to take a union bound over all candidate maps, 
even if we just apply a super trivial upper bound on the number of candidate maps with a fixed image. 

\medskip

\paragraph{The near-critical regime.} (\Cref{sec:near-critical}) 
Near the threshold, the preceding union bound becomes too crude, especially for retractions whose images contain almost every vertex. 
We therefore condition on a more detailed description of the typical random graph. 
These typical properties first imply that the boundary of a retraction image must be small. 
If a vertex outside the image has two neighbours inside it, then those vertices, together with its image under the retraction, form a $4$-cycle. 
Since few vertices lie on such cycles, only a small number of outside vertices can have several neighbours in the image.

The main difficulty is to rule out retractions with very large images; the other cases can be ruled out due to the typical properties.  
Unlike the proof strategy for the above-critical regime, a trivial bound on the number of candidate maps does not suffice. 
The property that the minimum degree is at least $2$ with high probability, however, allows us to design \emph{witnesses} that narrows down the count. 
Suppose we have a fixed image $S$. 
For every vertex outside the image $S$, once the random graph is fully revealed, choose two incident edges in a canonical way, whose existence is guaranteed by the minimum-degree condition. 
The resulting sparse witness substantially restricts the candidate map: after the graph on $S$ has been revealed, the map must send the part of the witness outside $S$ homomorphically into the revealed graph. 
A spanning-forest argument then gives a much sharper count of the possible maps than the trivial bound. 
Combining this count with the probabilities that the witness edges are present and that all edges forbidden by the retraction are absent yields a desired bound.

\section{Preliminaries}

\subsection{Basic definitions}

$[n]\defeq\{1,2,\dots,n\}$, and $\log$ is the natural logarithm. 

Unless stated otherwise, every graph in this manuscript is finite, simple, loopless and undirected. 
Additionally, every $n$-vertex graph is vertex-labelled with $[n]$. 
Denote by $N_G(v)$ the set of neighbours of $v$ in $G$. 
The \emph{degree} of a vertex $v$ in a graph $G$ is $\deg_G(v)=|N_G(v)|$, the number of neighbours of $v$ in $G$. 
The subscripts for $N_G(v)$ and $\deg_G(v)$ shall be dropped when it is clear from the context. 
The (maximum) \emph{degree} $\Delta(G)$ of a graph $G$ is the maximum degree of any vertex in $G$. 
The \emph{minimum degree} $\delta(G)$ of a graph $G$ is the minimum degree of any vertex in $G$. 
Specially, $\delta(G)=0$ if $G$ contains any isolated vertex. 
The \emph{clique number} $\omega(G)$ of a graph $G$ is the number of vertices of its largest clique.\footnote{Be aware of the notation overload of $\omega$. However, it is clear from the context whether it means the clique number or the asymptotic notation.}

Suppose a random variable $X\sim\mathrm{Bin}(n,p)$ follows the binomial distribution with expectation $\mu=np$. 
The (absolute form of) Chernoff bound states
\begin{equation} \label{eq:chernoff-bound}
\begin{gathered}
\Pr[X\geq t]\leq\left(\frac{\mathrm{e}\mu}{t}\right)^t \qquad \text{for any $t\geq \mu$};\\
\Pr[X\leq t]\leq \exp\{-\mu\}\left(\frac{\mathrm{e}\mu}{t}\right)^t \qquad \text{for any $t\leq \mu$}. 
\end{gathered}
\end{equation}

\subsection{Graph morphisms}

Here we collect basic definitions of graph morphisms. 

An \textit{isomorphism} is a bijection map $g\colon V(G)\to V(H)$ such that
\[
xy\in E(G) \iff g(x)g(y)\in E(H).
\]
An \textit{automorphism} of a graph $G$ is an isomorphism from $G$ to itself. 

A \textit{homomorphism} from $G$ to $H$ is a map $f\colon V(G)\to V(H)$ such that
\[
xy\in E(G) \implies f(x)f(y)\in E(H).
\]
An \textit{endomorphism} of a graph $G$ is a homomorphism from $G$ to itself. 

A \textit{retraction} $r$ from a graph $G$ to its induced subgraph $G[S]$ is an endomorphism such that
\[
r(V(G))=S \qquad \text{and} \qquad r(s)=s\text{ for every }s\in S. 
\]
It is \textit{proper} if $S\neq V(G)$. 

A standard fact is the following; see {\cite[Proposition 1.31]{MR2089014}}. 
\begin{proposition}[homomorphic core] \label{prop:hom-core}
The two statements are equivalent for any graph $G$. 
\begin{itemize}
\item[(1)] Every endomorphism of $G$ is an automorphism. 
\item[(2)] $G$ has no proper retraction. 
\end{itemize}
A graph $G$ is called a \emph{homomorphic core} if it satisfies (any of) the above two statements.
\end{proposition}

\begin{proof}
The direction (1)$\Rightarrow$(2) is trivial: any retraction by definition is an endomorphism. 

To show (2)$\Rightarrow$(1), assume $G$ has an endomorphism to a proper induced subgraph. 
Let $H$ be the induced subgraph with the minimum number of vertices that $G$ is homomorphic to. 
Thus by the choice of $H$, any endomorphism of $H$ must be an automorphism. 
Now let $f\colon V(G)\to V(H)$ be a homomorphism from $G$ to $H$, and $g\colon V(H)\to V(H)$ be the restriction of $f$ on $H$, i.e., $f(s)=g(s)$ for any $s\in V(H)$. 
Trivially, $g$ is an endomorphism, and hence an automorphism, and thus it has an inverse automorphism $g^{-1}$. 
Then $g^{-1}\circ f$ is a retraction of $G$ to $H$. 
\end{proof}

\subsection{Random graphs}

The Erd\H{o}s--R\'{e}nyi random graph model $\mathcal{G}(n,p)$ \cite{MR125031} is a distribution over $n$-vertex graphs with each pair of distinct vertices joined by an edge with probability $p$ independently. 
Erd\H{o}s and R\'{e}nyi \cite{MR130187} located the phase transition threshold for minimum degree $2$: 

\begin{theorem}[minimum degree phase transition]\label{thm:min-deg-phase-transition}
Let $p=p(n)\leq 1/2$. 
If
\[
\lim_{n\to+\infty}(np-\log n-\log\log n)=+\infty, 
\]
then w.h.p. $\delta(G)\geq 2$. 
Conversely, if
\[
\lim_{n\to+\infty}(np-\log n-\log\log n)=-\infty, 
\]
then w.h.p. $\delta(G)\leq 1$. 
\end{theorem}

\begin{proof}
Set $d=np$, and let $X$ be the number of vertices of degree at most $1$. 
For every vertex $v$,
\[
\Pr[\deg(v)\leq 1]=(1-p)^{n-1}+(n-1)p(1-p)^{n-2}\leq (1+d)\exp\{-d+1\}.
\]
By linearity of expectation, $\mathbb{E}[X]\leq n(1+d)\exp\{-d+1\}$. 
Suppose first that
\[
w_n\defeq d-\log n-\log\log n\longrightarrow+\infty.
\]
If $d\leq 2\log n$, then
\[
\mathbb{E}[X]\leq n(1+d)\exp\{-d+1\}=\exp\{1-w_n\}\frac{1+d}{\log n}.
\]
If $d>2\log n$, then, since $(1+x)\exp\{-x\}$ is decreasing for $x>0$,
\[
\mathbb{E}[X]\leq n(1+2\log n)\exp\{-2\log n+1\}=O\left(\frac{\log n}{n}\right).
\]
Since both right-hand-sides of the above two inequalities are $o(1)$, we have $\mathbb{E}[X]=o(1)$. 
Thus, by Markov's inequality, $X=0$ with high probability, and hence $\delta(G)\geq 2$.

Conversely, we apply a second-moment analysis. 
Suppose that
\[
w_n\defeq \log n+\log\log n-d\longrightarrow+\infty.
\]
We may assume $d\geq(\log n)/2$; this stochastically dominates the case when $d$ is smaller than this. 
Because $d=O(\log n)$, we have $p=d/n=o(1)$ and $np^2=d^2/n=o(1)$. 
Therefore,
\[
\mathbb{E}[X]=n(1-p)^{n-2}\bigl((1-p)+(n-1)p\bigr)=(1+o(1))n(1+d)\exp\{-d\}=(1+o(1))\frac{1+d}{\log n}\exp\{w_n\}\to +\infty. 
\]

For the second-moment bound, fix two distinct vertices $u$ and $v$. 
Conditioning on whether $uv$ is an edge gives
\begin{align*}
\Pr[\deg(u)\leq 1,\deg(v)\leq 1] 
&=(1-p)
\left((1-p)^{n-2}+(n-2)p(1-p)^{n-3}\right)^2
+p(1-p)^{2n-4}\\
&\leq 
\left((1-p)^{n-2}+(n-2)p(1-p)^{n-3}\right)^2.
\end{align*}
Moreover,
\begin{align*}
\Pr[\deg(v)\leq 1]
={}&(1-p)\left((1-p)^{n-2}
 +(n-2)p(1-p)^{n-3}\right)
 +p(1-p)^{n-2} \\
\geq{}&(1-p)\left((1-p)^{n-2}
 +(n-2)p(1-p)^{n-3}\right).
\end{align*}
Consequently,
\[
\Pr[\deg(u)\leq 1,\deg(v)\leq 1]
\leq
\frac{\Pr[\deg(v)\leq 1]^2}{(1-p)^2}
=(1+o(1))\Pr[\deg(v)\leq 1]^2,
\]
where the last equality follows from $p=o(1)$.

It follows that
\begin{align*}
\mathbb{E}[X^2]
&=\mathbb{E}[X]
 +\sum_{u\neq v}
 \Pr[\deg(u)\leq 1,\deg(v)\leq 1] \\
&\leq \mathbb{E}[X]
 +(1+o(1))n(n-1)
 \left(\frac{\mathbb{E}[X]}{n}\right)^2 \\
&\leq \mathbb{E}[X]+(1+o(1))\mathbb{E}[X]^2.
\end{align*}
Since $\mathbb{E}[X]\to\infty$, the second-moment inequality yields
\[
\Pr[X>0]
\geq
\frac{\mathbb{E}[X]^2}{\mathbb{E}[X^2]}
=1-o(1).
\]
Therefore $\delta(G)\leq 1$ with high probability.
\end{proof}

\section{The above-critical regime} \label{sec:above-critical}

With a multiplicative gap away from the phase transition threshold in \Cref{thm:main}, a (comparably) simpler proof can be obtained. 

\begin{theorem} \label{thm:main-weakened}
Let $p=p(n)\leq 1/2$ with
\[
\lim_{n\to +\infty}\frac{np}{\log n}=+\infty. 
\]
Then $G\sim\mathcal{G}(n,p)$ is a homomorphic core w.h.p.
\end{theorem}

\subsection{Typical structures}
Recall the following crude estimations on the degree and the clique number of a random graph in this regime. 
\begin{lemma} \label{lem:weak-typical-event}
In the same regime of $p=p(n)$ as in \Cref{thm:main-weakened}, the following holds simultaneously w.h.p. for a randomly drawn $\mathcal{G}(n,p)$:
\begin{itemize}
\item[(a)] (maximum degree) $\Delta(G)\leq 7n/12$; 
\item[(b)] (clique number) $\omega(G)\leq\lceil 4\log n/(\log 1/p)\rceil$. 
\end{itemize}
\end{lemma}

\begin{proof}
It suffices to show that each event fails with probability $o(1)$, and then apply union bound. 

For (a), note that every vertex degree is dominated by the binomial distribution $\mathrm{bin}(n-1,1/2)$, and then apply Hoeffding's inequality. 

To show (b), for any $W\in\mathbb{N}_{\geq 1}$, by union bound, 
\begin{equation} \label{eq:3-1-t1}
\Pr[\omega(G)\geq W]\leq\binom{n}{W}p^{\binom{W}{2}}\leq\exp\left\{W\log\frac{\mathrm{e}n}{W}-\frac{W(W-1)\log 1/p}{2}\right\}.
\end{equation}
Now let $W=\lceil 4\log n/(\log 1/p)\rceil+1$, so we have $W-1\geq 4\log n/(\log 1/p)$. 
The above is further upper bounded by 
\[
\eqref{eq:3-1-t1}\leq \exp\{W(1-\log n-\log W)\}\leq\exp\{W(1-\log n)\}=o(1). \qedhere
\]
\end{proof}

\subsection{Conditioned on a typical structure}

Over a randomly drawn $G\sim\mathcal{G}(n,p)$, define a ``global'' random event
\[
\mathcal{A}\qquad \defeq \qquad \Delta(G)\leq\frac{7n}{12}\qquad\text{and}\qquad\omega(G)\leq\left\lceil\frac{4\log n}{\log 1/p}\right\rceil. 
\]
By \Cref{lem:weak-typical-event}, $\Pr[\lnot\mathcal{A}]=o(1)$. 

For any map (candidate \textit{proper} retraction) $r\colon [n]\to S_r$ that is identity on $S_r\neq[n]$, define the random event
\[
\mathcal{R}_r\qquad \defeq \qquad \text{$r$ is a retraction}.
\]
By the definition of homomorphic cores (\Cref{prop:hom-core}), we need to show $\Pr\left[\bigvee_{r}\mathcal{R}_r\right]=o(1)$. 
But since $\Pr[\lnot\mathcal{A}]=o(1)$ and
\[
\Pr\left[\bigvee_{r}\mathcal{R}_r\right]
\leq\Pr\left[\left(\bigvee_{r}\mathcal{R}_r\right)\land \mathcal{A}\right]+\Pr[\lnot\mathcal{A}]
=\Pr\left[\bigvee_{r}(\mathcal{R}_r\land \mathcal{A})\right]+\Pr[\lnot A]
\leq\sum_{r}\Pr[\mathcal{R}_r\land \mathcal{A}]+\Pr[\lnot\mathcal{A}], 
\]
it suffices to prove that, summing over all proper $r$,  
\begin{equation}\label{eq:high-deg-goal}
\sum_{r}\Pr[\mathcal{R}_r\land \mathcal{A}]=o(1). 
\end{equation}

We first bound each term $\Pr[\mathcal{R}_r\land \mathcal{A}]$. 
Fix the map $r$ and let $G[S_r]$ be the random induced subgraph of $G$ on $S_r$. 
By the law of total probability,
\begin{equation}\label{eq:total-probability-ra}
\Pr[\mathcal{R}_r\land \mathcal{A}]=\sum_{H}\Pr[\mathcal{R}_r\land \mathcal{A}\mid G[S_r]=H]\cdot\Pr[G[S_r]=H], 
\end{equation}
where the summation goes over all graphs $H$ on the vertex set $S_r$. 
Then observe that, assuming the global event $\mathcal{A}$ happens, the induced subgraph $G[S_r]$ must also satisfy
\[
\Delta(G[S_r])\leq\frac{7n}{12}\qquad\text{and}\qquad\omega(G[S_r])\leq\left\lceil\frac{4\log n}{\log 1/p}\right\rceil, 
\]
and therefore, only terms with graphs $H$ satisfying the above condition survives in the summation of \eqref{eq:total-probability-ra}.  
Let $\mathcal{H}$ be the collection of all such $H$'s. \eqref{eq:total-probability-ra} can be rewritten as
\begin{equation}\label{eq:total-probability-ra-2}
\Pr[\mathcal{R}_r\land \mathcal{A}]=\sum_{H\in\mathcal{H}}\Pr[\mathcal{R}_r\land \mathcal{A}\mid G[S_r]=H]\cdot\Pr[G[S_r]=H]. 
\end{equation}
Dropping $\mathcal{A}$ yields an upper bound: 
\begin{equation}\label{eq:total-probability-ra-3}
\Pr[\mathcal{R}_r\land \mathcal{A}]\leq \sum_{H\in\mathcal{H}}\Pr[\mathcal{R}_r\mid G[S_r]=H]\cdot\Pr[G[S_r]=H]\leq\max_{H\in\mathcal{H}}\Pr[\mathcal{R}_r\mid G[S_r]=H]. 
\end{equation}

\subsection{Probability of a candidate retraction} \label{sec:weak-candidate-retraction}
It is left for us to characterise the quantity
\[\Pr[\mathcal{R}_r\mid G[S_r]=H].\]
Fix the map $r$. 
A random graph $G\sim\mathcal{G}(n,p)$ can be viewed as drawn in two independent stages: reveal all the edges among $S_r$, and then reveal the remaining edges. 
The candidate $r$ is invalidated if and only if any remaining edge is mapped to a non-edge in $H$ under $r$. 
Let 
\[
B(H,r)\defeq\{uv\colon r(u)r(v)\notin E(H)\}\subseteq\binom{[n]}{2}
\]
be the set of edges forced to be absent. 
Then $\Pr[\mathcal{R}_r\mid G[S_r]=H]=(1-p)^{B(H,r)}$ and hence
\begin{equation}\label{eq:total-probability-ra-4}
\Pr[\mathcal{R}_r\land \mathcal{A}]\leq\max_{H\in\mathcal{H}}(1-p)^{B(H,r)}. 
\end{equation}
The quantity $B(H,r)$ has an explicit formula.

\begin{lemma}\label{lem:weak-b-count}
For every $s\in S_r$, define
\[
b_{s}\defeq\left|\left\{ x\in [n]\backslash S \colon r(x)=s \right\}\right| 
\]
the number of vertices mapped by $r$ to a vertex $s$ except itself. 
Then $B(H,r)=B_1(H,r)+B_2(H,r)$ where
\begin{align}
B_1(H,r)&\defeq \sum_{s\in S_r}b_s(|S_r|-\deg_H(s));\label{eq:b1}\\
B_2(H,r)&\defeq \binom{n-|S_r|}{2}-\sum_{st\in E(H)}b_sb_t.\label{eq:b2}
\end{align}
\end{lemma}

Intuitively speaking, $B_1$ and $B_2$ together count all edges forced to be absent to make $r$ a retraction. 
$B_1$ includes edges with exactly one endpoint in the image $S_r$, and $B_2$ includes edges outside $S_r$. 

\begin{proof}[Proof of \Cref{lem:weak-b-count}]
We first count $B_1(H,r)$. 
Consider $x\in [n]\backslash S_r$ with $r(x)=s$, and another vertex $t\in S_r$. 
Under $r$, the edge $xt$ would map to $st$, and therefore, if $st\notin E(H)$, then the edge $xt$ must be absent, or otherwise $r$ is no longer a homomorphism. 
For a fixed vertex $s\in S$, there are $|S_r|-\deg_H(s)$ vertices $t$ in $S_r$ for which $st$ is not an edge, including $t=s$. 
There are $b_s$ vertices outside $S_r$ that get mapped to $s$, so $B_1$ counts all edges between $S_r$ and $[n]\backslash S_r$ forced to be absent. 
All the other edges between $S_r$ and $[n]\backslash S_r$ are free. 

For $B_2(H,r)$, consider a pair of vertices $x\neq y\in [n]\backslash S_r$ with $r(x)=s$ and $r(y)=t$. 
The edge $xy$ is permitted if and only if $st\in E(H)$. 
There are $\binom{n-|s|}{2}$ pairs outside $|S_r|$. 
For each edge $st\in E(H)$, exactly $b_sb_t$ pairs have one endpoint mapped to $s$ and the other endpoint mapped to $t$ under $r$. 
These are the only pairs outside $S_r$ allowed to be edges, and hence $B_2$ counts all edges outside $S_r$ forced to be absent. 
\end{proof}

In the next two subsections, we bound $B(H,r)$ in two cases based on the image size $|S_r|$ of $r$. 
This partitions the summation \eqref{eq:high-deg-goal} into two parts. 
\begin{lemma}[large image] \label{eq:high-deg-goal-case-1}
We have
\[
\sum_{\substack{r:\\2n/3\leq |S_r|\leq n-1}}\max_{H\in\mathcal{H}}(1-p)^{B(H,r)}=o(1). 
\]
\end{lemma}

\begin{lemma}[small image] \label{eq:high-deg-goal-case-2}
We have
\[
\sum_{\substack{r:\\1\leq |S_r|<2n/3}}\max_{H\in\mathcal{H}}(1-p)^{B(H,r)}=o(1). 
\]
\end{lemma}

\Cref{thm:main-weakened} follows once the above two lemmata are proved. 

\subsection{Large image} \label{sec:weak-large-image}
This subsection proves \Cref{eq:high-deg-goal-case-1}. 
For any $H\in\mathcal{H}$, plugging the condition $2n/3\leq |S_r|\leq n-1$ and the upper bound of $\Delta(H)$ into \eqref{eq:b1} in \Cref{lem:weak-b-count}, we have
\[
B_1(H,r)=\sum_{s\in S_r}b_s(|S_r|-\deg_H(s)))
\geq\sum_{s\in S_r}b_s\left(\frac{2n}{3}-\frac{7n}{12}\right)=\frac{n}{12}\sum_{s\in S_r}b_s
=\frac{(n-|S_r|)n}{12}
\]
where the last equation is by the definition of $b_s$. 
Trivially, we have $B(H,r)\geq B_1(H,r)$ and hence
\begin{equation}\label{eq:high-deg-case-1-term}
(1-p)^{B(H,r)}\leq(1-p)^{(n-|S_r|)n/12}\leq\exp\left\{-\frac{(n-|S_r|)np}{12}\right\}.
\end{equation}

Consider the summation in \Cref{eq:high-deg-goal-case-1}. 
For a fixed size $2n/3\leq \ell\leq n-1$, the number of choices of a set $S\subseteq [n]$ with $|S|=\ell$ is $\binom{n}{\ell}$. 
Once $S$ is fixed, the number of maps that fixes $S$ and maps $[n]\backslash S$ into $S$ is $\ell^{n-\ell}$. 
Therefore, the number of retractions is at most $\binom{n}{\ell}\ell^{n-\ell}$. 
Summing \eqref{eq:high-deg-case-1-term} over all $\ell$ in this range gives
\begin{align*}
&\sum_{\substack{r:\\2n/3\leq |S_r|\leq n-1}}\max_{H\in\mathcal{H}}(1-p)^{B(H,r)}
=\sum_{2n/3\leq \ell\leq n-1}\exp\left\{-\frac{(n-\ell)np}{12}\right\}\binom{n}{\ell}\ell^{n-\ell}
\\
\overset{\text{(a)}}{=}&\sum_{1\leq k\leq n/3}\exp\left\{-\frac{knp}{12}\right\}\binom{n}{k}(n-k)^k
\overset{\text{(b)}}{\leq}\sum_{1\leq k\leq n/3}\exp\left\{-\frac{knp}{12}\right\}\left(\frac{\mathrm{e}n^2}{k}\right)^k\\
=&\sum_{1\leq k\leq n/3}\exp\left\{k\left(2\log n+1-\log k-\frac{np}{12}\right)\right\}\\
\leq&\sum_{k\geq 1}\exp\left\{k\left(2\log n+1-\frac{np}{12}\right)\right\}
=\frac{\exp\{2\log n+1-np/12\}}{1-\exp\{2\log n+1-np/12\}}\\
\overset{\text{(c)}}{=}&n^{-\omega(1)}=o(1).
\end{align*}
Here, (a) is the substitution $k=n-\ell$, (b) is due to the inequality $\binom{n}{k}(n-k)^k\leq\left(\frac{\mathrm{e}n^2}{k}\right)^k$, and (c) is because of the condition $\frac{np}{\log n}=\omega(1)$. 
This concludes the proof of \Cref{eq:high-deg-goal-case-1}. 

\subsection{Small image}
This subsection proves \Cref{eq:high-deg-goal-case-2}. 
The following Motzkin--Straus inequality \cite{MR175813} from combinatorial optimisation is later used. 
A proof is provided here for completeness. 

\begin{lemma}[Motzkin--Straus inequality] \label{lemma:MSInequality}
Let $H$ be a graph with clique number $\omega(H)$, together with non-negative weights $\{b_v\}_{v\in V(H)}$ on vertices satisfying $\sum_{v\in V(H)}b_v=k>0$. 
Then
\[
\sum_{uv\in E(H)}b_ub_v\leq\frac{k^2}{2}\left(1-\frac{1}{\omega(H)}\right). 
\]
\end{lemma}

\begin{proof}
It suffices to prove for $k=1$ by normalising all weights. 

Repeat the following step, until there does not exist any pair of vertices $x,y$ of positive weights $b_x,b_y$ that are non-adjacent. 
We may write
\[
\sum_{uv\in E(H)}b_ub_v=\left(\sum_{u\in N(x)}b_u\right)b_x+\left(\sum_{v\in N(y)}b_v\right)b_y+C
\]
where $C$ does not depend on $b_x$ or $b_y$. 
If $\sum_{u\in N(x)}b_u>\sum_{v\in N(y)}b_v$, we update $b_x\gets b_x+b_y$ and $b_y\gets 0$; otherwise, update $b_x\gets 0$ and $b_y\gets b_x+b_y$. 
This procedure strictly increase the total count $\sum_{uv\in E(H)}b_ub_v$.

The above procedure eventually halts because it strictly decreases the size of the positive support. 
And when it halts, the support is a clique $K_r$ of some size $r\leq\omega(H)$. 
Let $\{b'_u\}$ be the final weights.  
On this clique, the Cauchy--Schwarz inequality gives
\[
\sum_{uv\in E(H)}b_ub_v \leq \sum_{xy\in E(K_r)}b'_ib'_j
=\frac{1}{2}\left(1-\sum_{x\in V(K_r)} b_i^2\right)
\overset{\text{C--S}}{\leq}\frac{1}{2}\left(1-\frac{1}{r}\right)
\leq\frac{1}{2}\left(1-\frac{1}{\omega(H)}\right). \qedhere
\]
\end{proof}

Write $W=\left\lceil\frac{4\log n}{\log 1/p}\right\rceil$ so $\omega(H)\leq W$ for any $H\in\mathcal{H}$. 
We have
\begin{align*}
B_2(H,r)&=\binom{n-|S_r|}{2}-\sum_{st\in E(H)}b_sb_t
\overset{\text{(d)}}{\geq} \binom{n-|S_r|}{2}-\frac{(n-|S_r|)^2}{2}\left(1-\frac{1}{W}\right)\\
&=\frac{n-|S_r|}{2W}(n-|S_r|-W)\geq\frac{n^2}{36W}
\end{align*}
for all sufficiently large $n$ since $W=O(\log n)$. 
Here, (d) is due to \Cref{lemma:MSInequality}. 
Again, we trivially have $B(H,r)\geq B_2(H,r)$ and hence
\begin{equation}\label{eq:high-deg-case-2-term}
(1-p)^{B(H,r)}\leq(1-p)^{n^2/(36W)}\leq\exp\left\{-\frac{n^2p}{36W}\right\}.
\end{equation}

Consider the summation in \Cref{eq:high-deg-goal-case-2}. 
Trivially, the number of maps from $[n]$ to $[n]$ is $n^n$. 
Using \eqref{eq:high-deg-case-2-term}, we have
\[
\sum_{\substack{r:\\1\leq |S_r|<2n/3}}\max_{H\in\mathcal{H}}(1-p)^{B(H,r)}\leq n^n\exp\left\{-\frac{n^2p}{36W}\right\}=\exp\left\{-n\log n\left(\frac{np}{36W\log n}-1\right)\right\}.
\]
We claim that ${np}/({W\log n})=\omega(1)$, which immediately implies that the above quantity is $o(1)$ and therefore finishes the proof of \Cref{eq:high-deg-goal-case-2}. 
Indeed, for all $n$ such that $p(n)<\sqrt{n}/n$, we have $W\leq 8$, and so ${np}/({W\log n})\geq np/(8\log n)$. 
For all $n$ such that $p(n)\geq\sqrt{n}/n$, since $p(n)\leq 1/2$, we have $W\leq\lceil (4/\log 2)\log n\rceil\leq 8\log n$,  and so ${np}/({W\log n})\geq \sqrt{n}/(8\log^2 n)$. This gives
\[
\frac{np}{W\log n}\geq\min\left\{\frac{np}{8 \log n},\frac{\sqrt{n}}{8\log^2 n}\right\}. 
\]
Both terms are $\omega(1)$ in the regime of $p$ as \Cref{thm:main-weakened}.

\section{The near-critical regime} \label{sec:near-critical}

For the rest of this manuscript, assume
\begin{equation} \label{eq:tight-condition}
\begin{gathered}
d\defeq np = \log n+\log\log n+w \qquad\text{ with }\qquad \lim_{n\to+\infty}w(n)=+\infty, \qquad\text{and} \\
d \leq \log^2 n,
\end{gathered}
\end{equation}
as otherwise \Cref{thm:main-weakened} applies. 

\subsection{Typical structure near the critical threshold}
We need some more standard facts about certain graph parameters in the above regime. 

Recall the following basic definitions. 
The \emph{independence number} $\alpha(G)$ of a graph $G$ is the number of vertices of its largest independent set, or alternatively, the clique number of its complement graph. 
The \emph{chromatic number} $\chi(G)$ of a graph $G$ is the minimum number $q$ that $G$ has a proper $q$-vertex-colouring. 

\begin{lemma} \label{lem:near-critical-properties}
Under the condition \eqref{eq:tight-condition}, the following holds simultaneously w.h.p. for a randomly drawn $G\sim\mathcal{G}(n,p)$:
\begin{itemize}
\item[(a)] (minimum degree) $\delta(G)\geq 2$;
\item[(b)] (maximum degree) $\Delta(G)\leq 4d$;
\item[(c)] ($4$-cycle count) at most $4d^5$ vertices lie on a $4$-cycle;
\item[(d)] (chromatic number) $\chi(G)\geq d/(5 \log d)$;
\item[(e)] (subgraph edge density) with $r_0\defeq{30\log d}/{\log\log d}$, 
every $U\subseteq [n]$ with $|U|\leq 8/p$ satisfies $|E(G[U])|<r_0|U|/2$. 
\end{itemize}
\end{lemma}

\begin{proof}
By a union bound, it suffices to prove that the failure probability of each item is $o(1)$.

\begin{itemize}
\item [(a)] Exactly the same as \Cref{thm:min-deg-phase-transition}. 
\item [(b)] This is derived from the union bound and Chernoff bound \eqref{eq:chernoff-bound}: 
\[
\Pr[\Delta(G)\geq 4d]\leq n\cdot\Pr[\deg_G(v)\geq 4d]\leq n\cdot 0.08^{d}\leq n\cdot 0.08^{\log n}=o(n^{-1.52})=o(1).
\]
\item [(c)] By the linearity of expectation, the expected number $Z$ of $4$-cycles is $p^4\cdot n(n-1)(n-2)(n-3)/8\leq d^4/8$. 
By Markov's inequality, the probability $\Pr[Z>d^5]\leq 1/(8d)=o(1)$, where at most $4d^5$ vertices lie on a $4$-cycle. 
\item [(d)] The expected number of independent sets of size $k$ is
\[
\binom{n}{k}(1-p)^{\binom{k}{2}}\leq\exp\left\{k\log\frac{\mathrm{e}n}{k}-\frac{pk(k-1)}{2}\right\}. 
\]
For any $k\geq 5\log d/p$, the above quantity tends to $0$, meaning w.h.p.~the independence number $\alpha(G)<5\log d/p$, and therefore the chromatic number $\chi(G)\geq d/(5 \log d)$.
\item [(e)] Let $u=|U|\leq 8/p$. 
The number of edges $|E(G[U])|$ is subject to the binomial distribution $\mathrm{Bin}(u(u-1)/2,p)$. 
By a union bound over all $U$ of size $u$, the probability that (e) is violated by a fixed $u$ is at most
\[
\binom{n}{u}\Pr\left[\mathrm{Bin}\left(\frac{u(u-1)}{2},p\right)\geq\frac{r_0 u}{2}\right]. 
\]
By \eqref{eq:chernoff-bound}, this quantity is bounded above by
\begin{equation} \label{eq:near-critical-typical-proof-5}
\binom{n}{u}\exp\left\{\frac{r_0 u}{2}\log\frac{\mathrm{e}pu}{r_0}\right\}\leq\exp\left\{u\left(\log\frac{\mathrm{e}d}{pu}+\frac{r_0}{2}\log\frac{\mathrm{e}pu}{r_0}\right)\right\}. 
\end{equation}
Here, the last inequality due to $\binom{n}{u}\leq(\mathrm{e}n/u)^u$ and $d=np$. 
The function $\log(\mathrm{e}d/x)+r_0/2\cdot\log(\mathrm{e}x/r_0)$ is increasing since $r_0>2$. 
At $x=8$, its value is bounded by
\[
\log\frac{\mathrm{e}d}{8}-\frac{r_0}{2}\log\frac{r_0}{8\mathrm{e}}\leq\log\frac{\mathrm{e}d}{8}-\frac{r_0}{4}\log\log d\leq -6\log d
\]
for sufficiently large $d$. 
As $pn\leq 8$, we upper bound \eqref{eq:near-critical-typical-proof-5} by
\[
\leq\exp\{-6u\log d\}=d^{-6u}.
\]
Summing all possible $u$ yields a bound $\sum_{u=1}^{8/p}d^{-6u}\leq d^{-6}=o(1)$ on the probability that (e) is violated. \qedhere
\end{itemize}
\end{proof}

An interesting consequence of the above typical graph parameters, especially Item (a)(b)(c), is the boundary size of the image $S_r$ of a retraction $r$. 
\begin{lemma} \label{lem:critical-boundary}
Let $r$ be a retraction to $S_r$. 
If events (a)(b)(c) in \Cref{lem:near-critical-properties} happen, then the number of edges between $S_r$ and $[n]\backslash S_r$ is upper bounded by $16d^6+n-|S_r|$. 
\end{lemma}

\begin{proof}
By \Cref{lem:near-critical-properties}(a), every vertex $x\in[n]\backslash S_r$ has at least two distinct neighbours. 
Suppose $x$ has two distinct neighbours $u,v\in S$. 
By the definition of retraction, the edges $r(x)r(u)=r(x)u$ and $r(x)r(v)=r(x)v$ exist, and $r(x)\neq x,u,v$. 
Therefore, the $4$-cycle $x\to u\to r(x)\to v\to x$ exist, so $x$ lies on a $4$-cycle. 
Consequently, any vertex in $[n]\backslash S_r$ that does not lie on any $4$-cycle must have at most one neighbours in $S$. 
Trivially, by \Cref{lem:near-critical-properties}(b), any vertex in $[n]\backslash S_r$ that lies on a $4$-cycle has at most $4d$ neighbours in $S$. 
The lemma then follows by \Cref{lem:near-critical-properties}(c). 
\end{proof}

\subsection{Conditioned on a typical structure}
In a similar manner as the proof of \Cref{thm:main-weakened}, with $G\sim\mathcal{G}(n,p)$ define the ``global'' event
\[
\mathcal{B}\qquad\defeq\qquad\text{All items in \Cref{lem:near-critical-properties} happen simutaneously}. 
\]
By \Cref{lem:near-critical-properties}, we have $\Pr[\mathcal{B}]=1-o(1)$. 

Like in \Cref{sec:weak-candidate-retraction}, we subdivide the event that a proper retraction exists into three cases depending on the size of the retraction's image. 
For any map (candidate proper retraction) $r\colon [n]\to S_r$ that is identity on $S_r\neq[n]$, we call the map is of Type I, II or III, based on the following criteria. 
\begin{align*}
1       \quad\leq\quad&   |[n]\backslash S_r|     \quad<   \quad   64d^6;\tag{Type I}\\
64d^6   \quad\leq\quad&   |[n]\backslash S_r|     \quad\leq\quad   n-8/p;\tag{Type II}\\
n-8/p     \quad<   \quad&   |[n]\backslash S_r|     \quad\leq\quad   n-1.\tag{Type III}
\end{align*}
We assume $n$ large enough so the above inequalities are well-defined. 
Define $\mathcal{R}_r$ to be the random event that $r$ is a retraction. 
For $\star\in\{\mathrm{I,II,III}\}$, define the events
\[
\mathcal{E}_{\star}\qquad\defeq\qquad\text{There is an $r$ of type $\star$ that $\mathcal{R}_r$ happens}. 
\]
Similar to \eqref{eq:high-deg-goal}, we only need to bound
\begin{equation}\label{eq:critical-goal}
\left(\sum_{\substack{r\colon\text{\textrm{\textnormal{Case I}}}}}\Pr[\mathcal{R}_r\land \mathcal{B}]\right)+\Pr[\mathcal{E}_{\mathrm{II}}\land \mathcal{B}]+\Pr[\mathcal{E}_{\mathrm{III}}\land \mathcal{B}]=o(1). 
\end{equation}

\begin{lemma} \label{lem:critical-goal-1}
We have
\begin{equation}\label{eq:goal-part-1}
\sum_{\substack{r\colon\text{\textrm{\textnormal{Case I}}}}}\Pr[\mathcal{R}_r\land \mathcal{B}]=o(1). 
\end{equation}
\end{lemma}

\begin{lemma} \label{lem:critical-goal-2}
We have $\Pr[\mathcal{E}_{\mathrm{II}}\land \mathcal{B}]=o(1)$. 
\end{lemma}

\begin{lemma} \label{lem:critical-goal-3}
We have $\Pr[\mathcal{E}_{\mathrm{III}}\land \mathcal{B}]=o(1)$. 
\end{lemma}

The first half of \Cref{thm:main} follows after the above three lemmata. 
The second half of \Cref{thm:main} is implied by \Cref{thm:min-deg-phase-transition}.

\subsection{Case I}
This subsection proves \Cref{lem:critical-goal-1} where the image of the candidate retraction is very large.

We begin with changing the order of summation in \eqref{eq:goal-part-1}. 
For any $S\subsetneq[n]$, if If $\mathcal{B}$ happens, then $\Delta(G[S])\leq 4d$ for any induced subgraph $G[S]$ of $G$. 
Let $\mathcal{H}_S$ be the collection of all
graphs $H$ on the vertex set $S$ such that $\Delta(H)\leq 4d$.
Analogously to \eqref{eq:total-probability-ra-2}, we have
\begin{equation}\label{eq:case-1-s1}
\begin{aligned}
\sum_{\substack{r\colon\text{\textrm{\textnormal{Case I}}}}}
\Pr[\mathcal{R}_r\land\mathcal{B}]
&\leq
\sum_{1\leq\ell<64d^6}
\ \sum_{\substack{S\subseteq[n]\\ |S|=n-\ell}}
\ \sum_{H\in\mathcal{H}_S}
\Pr[G[S]=H]\cdot Z(H)\\
&\leq
\sum_{1\leq\ell<64d^6}
\ \sum_{\substack{S\subseteq[n]\\ |S|=n-\ell}}
\max_{H\in\mathcal{H}_S}Z(H)
\end{aligned}
\end{equation}
with the definition
\begin{equation}\label{eq:zh-def}
Z(H)\defeq
\sum_{\substack{r\colon\\ S_r=S}}
\Pr[\mathcal{R}_r\land\mathcal{B}\mid G[S]=H]
\end{equation}
for $H\in\mathcal{H}_S$.

Fix $S$ and $H\in\mathcal{H}_S$. From this point onward, $H$ is fixed rather than random. 
Conditional on $G[S]=H$, all edges that are not fully in $S$ remain mutually independent. 

Suppose that $\mathcal{R}_r\land\mathcal{B}$ occurs for some candidate map $r$ with $S_r=S$. Since $\delta(G)\geq2$ under
$\mathcal{B}$, choose two distinct edges incident with each vertex $x\in[n]\backslash S$, and let $F$ be the simple graph formed by the union of all these chosen edges. 
For definiteness, one may choose the two edges whose other endpoints have the smallest labels. 
Formally, the edge set of $F$ is
\[
\bigcup_{x\in[n]\backslash S}\{xy_1,xy_2\colon\text{$y_1\neq y_2$ are the smallest indices that $xy_1,xy_2\in E(G)$}\}.
\]
Thus $F$ is determined by $G$, although in the estimate below we simply take a union bound over all possible $F$.
The graph $F$ is served as a \emph{witness}. 

As in the notation above, partition
\[
F_{\mathrm{out}}\defeq F[[n]\backslash S],
\qquad
F_{\mathrm{boundary}}\defeq F-F_{\mathrm{out}}.
\]
Set
\[
e_2\defeq |E(F_{\mathrm{out}})|,
\qquad
e_1\defeq |E(F_{\mathrm{boundary}})|. 
\]
So $F_{\mathrm{out}}$ is an \emph{induced} subgraph of $F$ that collects edges in $[n]\backslash S$, 
and $F_{\mathrm{boundary}}$ is a bipartite subgraph collecting edges between $S$ and $[n]\backslash S$. 
Let $\kappa$ be the number of connected components of
$F_{\mathrm{out}}$, including isolated vertices.

\begin{lemma}\label{lem:witness-count}
Any witness $F$ constructed above satisfies
\[
e_1+e_2\leq2(n-|S|)
\qquad\text{and}\qquad
e_1+e_2-\kappa\geq\frac{2(n-|S|)}{3}.
\]
\end{lemma}

\begin{proof}
The first inequality follows because two edges were selected at each vertex
of $[n]\backslash S$, and repeated selections produce only one edge in
the simple graph $F$.

Consider a component $J$ of $F_{\mathrm{out}}$. Suppose that $J$
has $a$ vertices and $b$ edges, and that $c$ witness edges join
$V(J)$ to $S$. Every vertex of $J$ is incident with at least two
edges of $F$, so
\[
2b+c\geq2a.
\]
If $J$ is a tree, then $b=a-1$, and hence $c\geq2$. It follows that $b+c-1\geq a$. 
If $J$ is not a tree, then $b\geq a$; moreover $a\geq3$, because the graph is simple. 
Therefore $b+c-1\geq a-1\geq2a/3$. 
Summing these inequalities over all $\kappa$ components gives the desired bound. 
\end{proof}

If the subgraph $G[S]$ has been revealed to be $H$, 
then a retraction $r$ restricted on $[n]\backslash S$ is a homomorphism from $F_{\mathrm{out}}$ to $H$,   
since $r(x)r(y)\in E(H)$ for any $xy\in E(F_{\mathrm{out}})$. 
This significantly reduces the number of candidate maps if the $F_{\mathrm{out}}$ part of the witness is fixed. 

To rewrite the summation for $Z(H)$ in \eqref{eq:zh-def}, we enumerate the possible witnesses and candidate maps. 
It is important that $H$ has already been fixed before this enumeration.

First choose $F_{\mathrm{out}}$. If it has $e_2$ edges, the number of
choices is at most
\[
\binom{\binom{n-|S|}{2}}{e_2}\leq(n-|S|)^{2e_2}.
\]

Next, for a fixed $F_{\mathrm{out}}$, we claim that the number of maps from $[n]\backslash S$ to $S$ that are
homomorphisms from $F_{\mathrm{out}}$ to $H$ is at most
\[
|S|^{\kappa}(\Delta(H))^{n-|S|-\kappa}\leq |S|^{\kappa}(4d)^{n-|S|-\kappa}. 
\]
To see this, fix one rooted spanning tree in each component of $F_{\mathrm{out}}$. 
Choosing the image $r(x)\in S$ of one root $x\in F_{\mathrm{out}}$ in each component gives $|S|^{\kappa}$ so many options. 
Each of the remaining vertex in $F_{\mathrm{out}}$ must map to a neighbour of its parent's image. 
This gives at most $\Delta(H)$ options, and $\Delta(H)\leq 4d$ because $\mathcal{B}$ happens. 

Now fix one of these compatible maps $r$. We next choose $F_{\mathrm{boundary}}$. 
A pair $xv$ where $x\in[n]\backslash S$ and $y\in S$ can belong to $F_{\mathrm{boundary}}$ only if $r(x)v\in E(H)$. 
Otherwise the present edge $xv$ would be mapped to a non-edge. 
The number of such allowed pairs is at most
\[
\sum_{x\in[n]\backslash S}\deg_H(r(x))
\leq4d(n-|S|).
\]
Thus, if $F_{\mathrm{boundary}}$ has $e_1$ edges, there are at most
\[
(4d(n-|S|))^{e_1}
\]
choices for it.

It remains to charge the probabilities of the required present and absent edges. 
By \eqref{eq:b1}, the map $r$ forces at least
\[
B_1(H,r)\geq(n-|S|)(|S|-4d)
\]
edges between $S$ and $[n]\backslash S$ to be absent. 
These forced-absent edges are disjoint from the witness edges in $F$: 
indeed, the edges of $F_{\mathrm{out}}$ have both endpoints outside $S$, 
and every edge $xv$ of $F_{\mathrm{boundary}}$ was chosen only when $r(x)v\in E(H)$, meaning the edge is not forced to be absent. 
Therefore, conditional on $G[S]=H$, the independence of the remaining edges gives
\[
p^{e_1+e_2}(1-p)^{(n-|S|)(|S|-4d)}
\]
as an upper bound on the probability that all witness edges are present and all these $B_1(H,r)$ edges are absent. 
We discard any additional conditions required for $r$ to be a retraction, which only enlarges the probability.

Combining the preceding enumeration steps and applying a union bound over
all witnesses gives
\begin{equation}\label{eq:small-witness-sum}
\begin{aligned}
Z(H)
&\leq (1-p)^{(n-|S|)(|S|-4d)}
\sum_{\substack{e_1,e_2\geq0,\ 1\leq\kappa\leq n-|S|\\
e_1+e_2\leq2(n-|S|)\\
e_1+e_2-\kappa\geq2(n-|S|)/3}}
(n-|S|)^{2e_2+e_1}|S|^\kappa
(4d)^{n-|S|-\kappa+e_1}p^{e_1+e_2}.
\end{aligned}
\end{equation}
Here, the summation over $e_1,e_2,\kappa$ is truncated by \Cref{lem:witness-count}. 

We finish by summing \eqref{eq:small-witness-sum}. In
\eqref{eq:case-1-s1}, $|S|=n-\ell$ with
$1\leq\ell<64d^6$. Since $p=d/n$ and $d\leq\log^2n$,
\[
p(\ell+4d)=O(d^7/n)=o(1),
\]
and hence
\begin{equation}\label{eq:small-cross-penalty}
(1-p)^{\ell(n-\ell-4d)}
\leq\exp\{-p\ell(n-\ell-4d)\}
=\exp\{-d\ell+o(\ell)\}.
\end{equation}
For every term in \eqref{eq:small-witness-sum},
\[
|S|^\kappa p^{e_1+e_2}
\leq n^\kappa\left(\frac dn\right)^{e_1+e_2}
\leq n^{-2\ell/3}d^{e_1+e_2},
\]
where the last inequality uses $e_1+e_2-\kappa\geq2\ell/3$ from \Cref{lem:witness-count}. 
Furthermore,
\[
(4d)^{\ell-\kappa+e_1}\leq(4d)^{3\ell},
\qquad
d^{e_1+e_2}\leq d^{2\ell},
\qquad
\ell^{2e_2+e_1-\ell}\leq\ell^{3\ell}.
\]
There are at most $O(\ell^3)$ triples $(e_1,e_2,\kappa)$ in the summation. 
Using $\binom n\ell\leq(\mathrm en/\ell)^\ell$, $d=\log n+\log\log n+w$, and \eqref{eq:small-cross-penalty}, 
the total contribution in \eqref{eq:case-1-s1} from all $S$ with $n-|S|=\ell$ is at most
\[
O(\ell^3)
\left(C n^{-2/3}\ell^3d^5\mathrm{e}^{-w+o(1)}\right)^\ell
\leq
O(\ell^3)\left(C n^{-2/3}d^{23}\right)^\ell
\]
for an absolute constant $C$. 
Since $d\leq\log^2n$, this is $o(n^{-\ell/2})$ uniformly for $1\leq\ell<64d^6$. 
Summing over $\ell$ proves \eqref{eq:goal-part-1}.

\subsection{Case II}
This subsection proves \Cref{lem:critical-goal-2}. 
Recall the condition on the image size of a candidate retraction $64d^6\leq|[n]\backslash S_r|\leq n-8/p$. 
We assume $n$ large enough, so $3/(4p)\leq n/2\leq n-8/p$. 
This is possible in the regime of $p$ in \eqref{eq:tight-condition}. 

Suppose that $\mathcal{E}_{\mathrm{II}}\land\mathcal{B}$ occurs; that is, $\mathcal{B}$ happens and there is a proper retraction $r$ in Case II. 
By \Cref{lem:critical-boundary}, the number of edges between $S_r$ and $[n]\backslash S_r$ is at most $\ell+16d^6\leq 5\ell/4$ for large enough $n$.
Here, $\ell=n-|S_r|$.  
Since $\delta(G)\geq 2$, the sum of the degrees of the vertices in
$[n]\backslash S_r$ is at least $2\ell$, 
and so the number of edges of $G$ inside $[n]\backslash S_r$ is at least $3\ell/8$. 
Therefore, the occurrence of a Case-II retraction implies the existence of a set $U\subseteq[n]$ of size $64d^6\leq|U|\leq n/2$ such that
\begin{equation}\label{eq:case-2-set-properties}
X_1\defeq|E(G[U])|\geq\frac{3|U|}{8}
\qquad\text{and}\qquad
X_2\defeq|E_G(U,[n]\backslash U)|\leq\frac{5|U|}{4},
\end{equation}
where $E_G(U,[n]\backslash U)$ denotes the set of edges having one
endpoint in $U$ and the other endpoint in $[n]\backslash U$.
We show that this is unlikely to happen, and hence prove \Cref{lem:critical-goal-2}. 

First fix a set $U$ of size $\ell$. 
Note that the random edges in $E(G[U])$ and $E_G(U,[n]\backslash U)$ are mutually independent, 
and so the two counts $X_1\sim\mathrm{Bin}(\ell(\ell-1)/2,p)$ and $X_2\sim\mathrm{Bin}(\ell(n-\ell),p)$ are independent binomial distributions. 
To bound the probabilities of the two events in \eqref{eq:case-2-set-properties}, we apply the Chernoff bound \eqref{eq:chernoff-bound}. 

For $X_1$, because $\mathbb{E}[X_1]<p\ell^2/2$, we have
\[
\Pr\left[X_1\geq\frac{3\ell}{8}\right]\leq\left(\frac{\mathrm{e}p\ell^2/2}{3\ell/8}\right)^{3\ell/8}=\left(\frac{4\mathrm{e}p\ell}{3}\right)^{3\ell/8}
\]
by the upper tail bound when $\ell\leq 3/(4p)$. 
If this condition does not hold, then the right hand side is greater than $1$, so the probability is trivially bounded by $1$. 
For $X_2$, because $\mathbb{E}[X_2]=d\ell(1-\ell/n)\geq d\ell/2>5\ell/4$ for all sufficiently large $n$, we can always apply the lower tail bound and have
\begin{equation} \label{eq:case-2-x2}
\Pr\left[X_2\leq\frac{5\ell}{4}\right]\leq\exp\left\{-d\ell\left(1-\frac{\ell}{n}\right)\right\}\left(\frac{4\mathrm{e}d(1-\ell/n)}{5}\right)^{5\ell/4}. 
\end{equation}
As $X_1$ and $X_2$ are independent, the probability that the event in \eqref{eq:case-2-set-properties} happens for a fixed $U$ is at most the product of the above two bounds. 

By union bound, the probability of \eqref{eq:case-2-set-properties} happening for any $U$ is at most
\begin{align}
&\sum_{64d^6\leq \ell\leq 3/(4p)}\left(\frac{4\mathrm{e}p\ell}{3}\right)^{3\ell/8}\exp\left\{-d\ell\left(1-\frac{\ell}{n}\right)\right\}\left(\frac{4\mathrm{e}d(1-\ell/n)}{5}\right)^{5\ell/4}\label{eq:case-2-total-1}\\
+&\sum_{3/(4p)\leq \ell\leq n/2}\exp\left\{-d\ell\left(1-\frac{\ell}{n}\right)\right\}\left(\frac{4\mathrm{e}d(1-\ell/n)}{5}\right)^{5\ell/4}\label{eq:case-2-total-2}\\
+&\sum_{n/2\leq \ell\leq n-8/p}\exp\left\{-d\ell\left(1-\frac{\ell}{n}\right)\right\}\left(\frac{4\mathrm{e}d(1-\ell/n)}{5}\right)^{5\ell/4}.\label{eq:case-2-total-3}
\end{align}
The above summation is well-defined for large enough $n$; that is, $3/(4p)\leq n/2\leq n-8/p$ for $n$ large enough. 

To show \eqref{eq:case-2-total-1} is $o(1)$, perform a straightforward calculation to see that the summand is $(\mathrm{e}^{\star})^\ell$ where
\[
\star\leq 1.454+\frac{3}{8}\log p\ell +\frac{5}{4}\log d\left(1-\frac{\ell}{n}\right)-d\left(1-\frac{\ell}{n}\right)
\leq
1.347+\frac{5}{4}\log d-\frac{d}{2}.
\]
Here we use the condition $\ell\leq 3/(4p)\leq n/2$. 
This gives that \eqref{eq:case-2-total-1} is upper bounded by $e^{-d/3}=o(1)$ for sufficiently large $n$. 
The second term \eqref{eq:case-2-total-2} can be bounded by following almost the same lines. 

For the third term, note that the function $-x+\frac{5}{4}\log x$ is decreasing for $x\geq 5/4$, and the condition on $\ell$ gives $d(1-\ell/n)\geq 8>5/4$. 
Therefore, the summand in \eqref{eq:case-2-total-3} is $(\mathrm{e}^{\star})^\ell$ with
\[
\star=\frac{5}{4}\log\frac{4\mathrm{e}}{5}+\frac{5}{4}\log d\left(1-\frac{\ell}{n}\right)-d\left(1-\frac{\ell}{n}\right)\leq\frac{5}{4}\log\frac{4\mathrm{e}}{5}+\frac{5}{4}\log 8-8\leq -4.42, 
\]
and the summation in \eqref{eq:case-2-total-3} is therefore $o(1)$ as $n$ grows.

\subsection{Case III}
This subsection proves \Cref{lem:critical-goal-3}. 

A graph has \emph{degeneracy} $k$ if every non-empty subgraph has at least one vertex of degree at most $k$. 
As a standard fact, a $k$-degenerate graph has chromatic number at most $k+1$ by an induction on the number of vertices: 
by removing a vertex of degree at most $k$ and colouring the rest of the graph recursively, the removed vertex will always have at most $k$ neighbors already coloured, leaving at least one colour free out of $k+1$ total colours. 

Suppose one such retraction $r$ exists under $\mathcal{B}$. 
Recall the condition on the size of retraction $1\leq |S_r|<8/p$, i.e., it is extremely small. 
By \Cref{lem:near-critical-properties}(d), the chromatic number of the random graph $G$ satisfies is at most $d/(5\log d)$ for all sufficiently large $n$ under $\mathcal{B}$. 

On the other hand, \Cref{lem:near-critical-properties}(e) applies to every subset $U\subseteq S_r$. 
It shows that every non-empty induced subgraph of $G[S_r]$ has average degree strictly less than $r_0$, 
and therefore has a vertex of degree at most $\lfloor r_0\rfloor$. 
Consequently, the degeneracy of $G$ is at most $\lfloor r_0\rfloor$, meaning the graph $G[S_r]$ is $(\lfloor r_0\rfloor+1)$-colourable. 
Since
\[
\lfloor r_0\rfloor+1=\left\lfloor\frac{30\log d}{\log\log d}\right\rfloor+1<\frac{d}{5\log d}
\]
for all sufficiently large $n$, we obtain
\[
\chi(G[S_r])<\chi(G).
\]
This contradicts the existence of the retraction: 
the retraction $G\to G[S_r]$ implies $\chi(G)\leq\chi(G[S_r])$, 
while the fact that $G[S_r]$ is an induced subgraph trivially implies the reverse inequality. 
Hence a retraction would force $\chi(G)=\chi(G[S_r])$. 
This proves \Cref{lem:critical-goal-3}. 

\bibliographystyle{alphaurl}
\bibliography{refs}

\end{document}